\documentclass{amsart}

\usepackage{latexsym,amsmath,amsfonts,amscd,amssymb, amsthm}
\theoremstyle{plain}
\newtheorem{theorem}{Theorem}[section]
\newtheorem{lemma}{Lemma}[section]
\newtheorem{proposition}{Proposition}[section]
\newtheorem{corollary}{Corollary}[section]

\theoremstyle{definition}
\newtheorem{definition}{Definition}[section]
\newtheorem{example}{Example}[section]

\theoremstyle{remark}
\newtheorem{remark}{Remark}[section]

\title{Proper holomorphic maps from the unit disk to some unit ball}

\author{John P. D'Angelo}

\address{Dept. of Mathematics, Univ. of Illinois, 1409 W. Green St., Urbana IL 61801}

\email{jpda@math.uiuc.edu}

\author{Zhenghui Huo}

\address{Dept. of Mathematics, Univ. of Illinois, 1409 W. Green St., Urbana IL 61801}

\email{huo3@illinois.edu}

\author{Ming Xiao}

\address{Dept. of Mathematics, Univ. of Illinois, 1409 W. Green St., Urbana IL 61801}

\email{mingxiao@illinois.edu}

\begin{document}

\begin{abstract} We study proper rational maps from the unit disk to balls in higher dimensions.
After gathering some known results, we study the moduli space of unitary equivalence classes
of polynomial proper maps from the disk to a ball, and we establish a normal form for 
these equivalence classes. We also prove
that all rational proper maps from the disk to a ball are homotopic in target dimension at least $2$.

{\bf AMS Classification Numbers}: 32H35, 30J99, 51F25, 32M99.


\end{abstract}

\maketitle

\section{Introduction}

There is a vast literature on proper holomorphic mappings between balls in possibly different dimensional complex Euclidean spaces.
See for example [D1], [D2], [F], [H], [HJ], [JZ] and their references. See also [BEH], [EHZ], [L] and their references for related work
when the target hypersurface is a hyperquadric rather than a sphere.
Many of these results assume the domain dimension is at least two. In this paper we gather together various facts
when the domain is one-dimensional and prove several new results in this setting. In particular we compute the moduli space
of unitary equivalence classes of polynomial maps taking the unit circle to some unit sphere and we find a normal form.
See Theorem 4.1. To do so we analyze what we call upper-trace identities. We also give a new result on homotopy classes
in target dimension at least $2$.

Section 2 on singularities gathers together some known results when the source dimension is $1$.
Section 3 on equivalence relations helps clarify three notions arising in the study of proper mappings between balls: spherical equivalence, unitary equivalence, and homotopy equivalence. Section 4 includes the new results on unitary equivalence
for polynomial maps sending the unit circle in some unit sphere. Section 5 includes a proof that, 
when the target dimension is at least $2$, there is but one homotopy
equivalence class for rational proper maps from the disk to a ball.

All three authors thank the referee for noting several improvements and corrections.
The first author acknowledges support from NSF Grant DMS 13-61001.

\section{singularities}

In this section we consider singularities for
proper holomorphic maps from the unit disc $\Delta$ in $\mathbb{C}$ to the unit ball $\mathbb{B}_N$ in 
$\mathbb{C}^N$ for $N \ge 2$. The case when the target is one-dimensional is classical; such maps are precisely
the finite Blaschke products. Such maps are of course rational and extend holomorphically past the unit circle.
It is known ([CS]) that rational proper holomorphic maps from $\mathbb{B}_n$ to some $\mathbb{B}_N$ also extend
holomorphically past the sphere. Furthermore, if $n\ge 2$, and a proper holomorphic map $f:\mathbb{B}_n \to \mathbb{B}_N$ 
extends past the sphere, then it is rational ([F]). It is worth observing the following result; when the domain is the unit disk
and the target ball is higher dimensional, there are proper maps that extend past the circle but that are not rational.

\begin{proposition} For each $N\ge 2$, there is a non-algebraic holomorphic proper map from the open unit
disk $\Delta$ to $\mathbb{B}_N$ that extends holomorphically to a neighborhood of the closed unit disk.
\end{proposition}

\begin{proof}
Write $f_{1}(z)=e^{z-2}.$ Then $|f_1|=e^{x-2} < 1$ on the unit circle. 
Let $\phi(x)$ be the real analytic function 
in a neighborhood of $\overline{\Delta}$ 
such that $e^{2\phi(x)}+e^{2(x-2)}=1$ on the unit circle. 
Let $u(z)$ be the harmonic function  with boundary value $\phi$. 
Since $\phi$ is real analytic in a neighborhood $D$ of $\overline{\Delta}$, also $u$ is real-analytic there. Let
$v$ be the harmonic conjugate of $u$ in $D$ and put $h=u+iv$. Then $h$ and thus $f_2=e^{h}$ is holomorphic in $D$. Put $F=(f_1,f_2)$.
On the circle,
$$ ||F||^2 = |f_1|^2 + |f_2|^2 =  e^{2(x-2)} + e^{2\phi} = 1$$
and hence $F$ maps the unit circle to the unit sphere. Also, $F$ is not algebraic.
\end{proof}

We next observe the following simple fact about singularities of rational proper maps from
$\Delta$ to $\mathbb{B}_N$. 
\begin{proposition}
Let $F=\frac{(p_{1},..., p_{N})}{q}$ be a holomorphic rational proper map from $\Delta$ to $\mathbb{B}_N$. Then $F$ extends holomorphically across the unit circle.
\end{proposition}
\begin{proof} We may assume that $F$ is reduced to lowest terms. Suppose $q(z_0) = 0$ for some $z_0$ on the circle. Since $q$ is a polynomial, it is divisible by $(z-z_0)$.
Since $F$ is proper, $||F(z)||^2$ tends to $1$ as $z$ tends to $z_0$. Therefore $||p(z)||^2$ tends to $|q(z_0)|^2 = 0$
as $z$ tends to $z_0$. Hence each component of $p$ is also divisible by $(z-z_0)$, and $F$ is not reduced to lowest terms.
Thus $q$ does not vanish on the circle and the conclusion follows.
 \end{proof}

Let $z \in {\mathbb C}^n$. We use the term {\bf rational sphere map} for a rational function $z \mapsto {p(z) \over q(z)}$
that maps the unit sphere in the source to the unit sphere in the target.
We assume without loss of generality that $p$ and $q$ have no common factor and, when $n\ge 2$, that $q(0)=1$.
Each rational proper map between balls defines a rational sphere map; for $n\ge 2$ the only other examples 
are constants. When $n=1$, however, rational sphere maps also include maps that have singularities
in the unit disk. In particular the denominator can have a power of $z$ as a factor.
The following elementary result indicates that the denominator can vanish at points inside the disk; the $a_k$
in this lemma can be anywhere except on the circle itself.

\begin{lemma} Suppose ${p \over q}$ is a rational function mapping the unit circle
to the unit sphere in some ${\mathbb C}^N$, and ${p\over q}$ is reduced to lowest terms.
Then the denominator $q$ can be written
$$ q(z) = c z^m \prod_{k=1}^K (1 - {\overline a_k}z),  \eqno (1) $$
where $c$ is a constant, $m\ge 0$, and $|a_k| \ne 1$ for each $k$. 
\end{lemma}

\section{Equivalence relations for proper mappings}

We introduce various equivalence relations  for proper holomorphic mappings $f:{\mathbb B}_n \to {\mathbb B}_N$ between balls.
We first mention a convenient notation:
we often write maps of the form $(f,0)$ as $f \oplus 0$. Here we have not specified the number of zero components.

Because the automorphism group of the unit ball (in each dimension) is transitive (and hence large),
it is natural to consider {\it spherical equivalence} of proper maps. 

\begin{definition} Proper mappings $f,g$ from ${\mathbb B}_n$ to ${\mathbb B}_N$ are {\it spherically equivalent}
if there are automorphisms $\psi$ of the target ball and $\chi$ of the domain ball 
such that $f= \psi \circ g \circ \chi$. Proper maps are {\it unitarily equivalent}
if both of these automorphisms can be chosen to be unitary maps. \end{definition}

The unitary group ${\bf U}(n)$ is a subgroup of the automorphism group. It is natural in some problems 
to consider unitary equivalence rather
than spherical equivalence. 
Homotopy equivalence [DL1] also plays a crucial role. See also [DL2].

\begin{definition} Let $f:{\mathbb B}_n \to {\mathbb B_{N_f}}$  and $g:{\mathbb B}_n \to {\mathbb B_{N_g}}$
be proper holomorphic maps. We say $f$ and $g$ are {\it homotopic in target dimension} $N$
if there is a one-parameter family of proper maps $H_t:{\mathbb B}_n \to {\mathbb B}_N$ such that
$H_0 = f \oplus 0$ and $H_1 = g \oplus 0$. We assume the map $(z,t) \to H_t(z)$ is continuous.
\end{definition}

\begin{example}\label{rmsp}
Spherical equivalence implies homotopy equivalence. See [DL1] for the simple proof,
which relies on the explicit form of the automorphisms and the path-connectedness of
the unitary group. In one dimension we deform the map $e^{i\theta} {z-a \over 1 - {\overline a} z}$
to the identity by replacing $a$ with $(1-t)a$ and $\theta$ 
with $(1-t)\theta$. In Theorem 5.1 we prove 
that a holomorphic rational proper map from the disk to a ball and of
degree one is homotopic to $z \oplus 0$. 
\end{example}

We can also consider homotopy equivalence for rational sphere maps. For $n\ge 2$, we are including constant maps
as well as proper maps between balls. When $n=1$, by Lemma 2.1, we are also including functions with poles inside the unit disk.
For rational sphere maps that are holomorphic on the ball, thus in particular when $n\ge 2$,
the Taylor coefficients of a homotopy depend continuously on $t$.
When $n=1$, rational sphere maps with a pole at $0$ do not have a Taylor expansion at $0$.

The next lemma indicates how homotopy depends upon the target dimension.

\begin{lemma}If $f,g : {\mathbb B}_n \to {\mathbb B_{N}}$ are proper maps,
then $f\oplus 0$ and $g \oplus 0$ are homotopic in target dimension $N+n$. 
If $n \ge 2$ and $f,g$ are rational sphere maps, then 
 $f\oplus 0$ and $g \oplus 0$ are homotopic in dimension $n+1$. \end{lemma}
\begin{proof} Let $z$ denote the identity map. Note that $f \oplus 0$ is homotopic to $0 \oplus z$ in target dimension $N+n$ 
via the homotopy
$(\sqrt{1-t^2}\  f, tz)$. The same applies to $g \oplus 0$. Since homotopy
equivalence is an equivalence relation, $f \oplus 0$ and $g \oplus 0$ are homotopic
in target dimension $N+n$. 

If we regard $f$ and $g$ as rational sphere maps, thus allowing constant maps, the same argument works with $z$ replaced
by the constant function $1$. Thus $f\oplus 0$ and $g\oplus 0$
are homotopic in target dimension $N+1$.
\end{proof}

\begin{remark} When the source dimension exceeds $1$, the degree of a polynomial or rational proper map is {\bf not}
a homotopy invariant. See [DL1] for an explicit one-parameter family $H_t$
of polynomial maps from ${\mathbb B}_2$ to ${\mathbb B}_5$ such that
the embedding dimension of each $H_t$ is $5$ and yet the degree 
is not constant in $t$. \end{remark}

\section{unitary equivalence}

We now assume that the source dimension is $1$.
Theorem 4.1 provides a complete analysis of the polynomial case.
Consider a polynomial $f:{\mathbb C} \to {\mathbb C}^N$ that maps the circle to the unit sphere.
Thus $||f(z)||^2 = 1$ on the circle. This condition leads to a system of linear equations
for the inner products of the coefficient vectors. 
Since $U$ is unitary if and only if $U$ preserves all inner products,
this system of equations takes into account unitary equivalence in the 
target space. The identities (3) and (4) below play a major role in this paper.
They are special cases of the upper-trace identities from Definition 4.2.

Let $f:{\mathbb C} \to {\mathbb C}^N$ be a polynomial of degree $d$. If the coefficient vectors are linearly independent,
then $N\ge d+1$. Without loss of generality we will assume that $N=d+1$ when we have a polynomial of degree $d$.

\begin{proposition} \label{tqsystem}
Let $f:{\mathbb C} \to {\mathbb C}^{N}$ be a polynomial of degree $d$. For $A_j \in {\mathbb C}^N$ put
$$f(z) = \sum_{j=0}^d A_j z^j. \eqno (2) $$
Then $f$ maps the unit circle to the unit sphere if and only if
the inner products $B_{jk} = \langle A_j, A_k \rangle$ of the coefficient vectors $A_j$ satisfy the following linear system:
$$ \sum_{j=0}^d B_{jj} = 1 \eqno (3) $$
$$ {\rm For} \ l \ne 0, \ \ \ \sum_{k=0}^{d-l} B_{(k+l)k} = 0. \eqno (4) 
$$\end{proposition}

\begin{proof} On the circle we have the condition
$$ 1 = ||f(z)||^2 = ||\sum_{j=0}^d A_j z^j||^2 = \sum_{j,k} \langle A_j, 
A_k\rangle z^{j-k}.  \eqno (5) $$
Replace $z$ by $e^{i\theta}$ in (5) and equate Fourier coefficients. For each $l$ we get
$$  1 = \sum_k B_{(k+l)k} e^{il \theta}. $$
Equation (3) follows because the constant term equals $1$. Equation (4) follows 
from equating the coefficient of $e^{i l \theta}$ to $0$. 
\end{proof}

Note that a matrix of inner products is non-negative definite, and hence all of its principal minor determinants
are non-negative. We write $B(f)$ for the matrix of inner products of a polynomial sphere map $f$. 

\begin{definition} Let $B=(B_{jk})$ and $C=(C_{jk})$ be square Hermitian matrices of the same size.
We define an equivalence relation $ \sim $ by $B \sim C$ if there is a $\theta$ such that
$$ B_{j(j+m)} = e^{im\theta} C_{j(j+m)}  $$
for all $j$ and $m$. We say simply that $B$ and $C$ are *-equivalent. \end{definition}

Recall that $f$ and $g$ are unitarily equivalent if $f = U \circ g \circ u$ for some $u \in {\bf U}(1)$
and some $U \in {\bf U}(d+1)$.

\begin{corollary} Let $f$ and $g$ be proper polynomial maps 
from $\Delta$ to ${\mathbb B}_{d+1}$ of degree $d$.  Then
$f$ and $g$ are unitarily equivalent if and only if $B(f)$ is *-equivalent to $B(g)$.\end{corollary}

\begin{corollary} The moduli space of unitarily equivalent polynomial sphere maps of degree $d$
from the circle to $S^{2d+1}$ is the quotient space $W/ \sim $, where $W$ is the collection
of non-negative $(d+1)$-by-$(d+1)$ Hermitian matrices satisfying (3) and (4) above,
and $ \sim $ is defined as in Definition 4.1. \end{corollary}

There are ${d(d+1) \over 2}$ parameters, of which $d$ are real. 
If all the off-diagonal elements of $B$ are non-negative real numbers, then we can go no further using
the group ${\bf U}(1)$. 
Otherwise we can make a chosen non-zero off-diagonal element positive via a rotation in the source,
and thus we can make an additional parameter real.
The moduli space is a semi-algebraic real subset of ${\mathbb C}^{(d+1)^2}$.

We want to give normal forms for the unitary equivalence classes.
For small target dimension and/or low degree it is fairly easy to solve the linear system from Proposition 4.1.
After stating the general result,
we list the normal forms explicitly for $d\le 3$.

The idea is simple. We regard the coefficients as vectors in the target space and select an orthonormal basis
in which their matrix representation is as simple as possible. Proposition 4.1 shows that the condition
of mapping the circle to the sphere depends only on the inner products of these vectors, and hence this
matrix depends only upon unitary transformations in the target.

\begin{theorem} Let $f(z) = \sum_{j=0}^d A_j z^j$ be a polynomial sphere map of degree $d$. Without loss of generality
assume that $f: {\mathbb C} \to {\mathbb C}^{d+1}$. For each $d$, there is an orthonormal basis $e_0, ..., e_d$ 
of the target ${\mathbb C}^{d+1}$ in which
$$ A_j = \sum a_{jk}e_k $$
and these vectors (and hence $f$) are in the following  partial normal form
with respect to ${\bf U}(d+1)$: \end{theorem}

\begin{itemize} 
\item $A_0 = ||A_0|| e_0$ and $A_d = ||A_d|| e_d$. 

\item $a_{(d-1)0} = -\alpha ||A_d||$ for some complex parameter $\alpha$.

\item $a_{1d} = {\overline \alpha} ||A_0||$.

\item The other entries in the first and last row are complex parameters.

\item The {\bf inner matrix} $a_{jk}$ for $1 \le j,k \le d-1$ is upper-triangular.

\item The first $d-2$ diagonal elements of the inner matrix are non-negative real parameters; the last diagonal
entry is real, non-negative and determined. 
\end{itemize}

\begin{remark} There is no unique way to put the inner matrix into upper-triangular form, and hence our partial normal form is not unique. For small $d$, in the next propositions we provide the extra information needed for uniqueness. \end{remark}

Let ${\mathcal A}$ denote the matrix with column vectors $A_j$ in this basis.
Then ${\mathcal A}$ satisfies all the identities in (4).
In particular, ${\rm trace} ({\mathcal A}^* {\mathcal A}) = 1$. Hence the parameters
in Theorem 4.1 are constrained by various inequalities. For example, Proposition 4.4 below
gives precise ranges of values for these parameters in case $f$ is degree $2$.

\begin{proposition} Let $f$ be as in Theorem 4.1. Consider
the $(d+1)$ by $(d+1)$ normal form matrix ${\mathcal A}$ from Theorem 4.1. For $0\le d\le 3$
these matrices have the following explicit forms:

\begin{itemize}
\item When $d=0$, the normal form is $(1) = \left( ||A_0|| \right) $. (no parameters)
\item When $d=1$, the normal form is
$$ \begin{pmatrix} ||A_0|| & 0 \cr 0 & ||A_1|| \end{pmatrix} $$
where $||A_0||^2 + ||A_1||^2 = 1$. (one real parameter, either $||A_0||$ or $||A_1||$.)
\item When $d=2$ the normal form is 
$$ \begin{pmatrix} ||A_0|| & -\alpha ||A_2|| & 0  \cr 0 & a_{11} & 0 \cr 0 & {\overline \alpha} ||A_0|| & ||A_2|| \end{pmatrix}. $$
Here $\alpha$ is a complex parameter, $a_{11}$ is real and the relation (6) holds:
$$ \left(||A_0||^2 + ||A_2||^2 \right) \ \left(1 + |\alpha|^2\right) + |a_{11}|^2 = 1. \eqno (6) $$
Both $||A_0||$ and $||A_2||$ are real parameters and $\alpha$ is a complex parameter. (Using a source unitary we can make $\alpha$ real if we wish.) 
\item When $d=3$ the normal form is 
$$ \begin{pmatrix} ||A_0|| & a_{10} & -\alpha ||A_3|| & 0  \cr 0 & a_{11} & a_{21} & 0 \cr 0 & 0 & a_{22} & 0 \cr 0 & {\overline \alpha} ||A_0|| & a_{23} & ||A_3|| \end{pmatrix}. $$
Here $\alpha$, $a_{10}$, and $a_{23}$ are complex parameters, $||A_0||$, $||A_3||$, and $a_{11}$ are real parameters,
$a_{22}$ is determined and is real, 
and two equations hold:
$$ ||A_0|| {\overline {a_{10}}} - {\overline \alpha} a_{10} ||A_3|| + 
a_{11} {\overline {a_{21}}} + {\overline \alpha} ||A_0|| {\overline {a_{23}}} + a_{23} ||A_3|| = 0 \eqno (7.1) $$
$$ {\rm trace} ({\mathcal A}^* {\mathcal A}) = \sum |a_{jk}|^2 = 1.  \eqno (7.2) $$
If $a_{11} \ne 0$, then (7.1) determines $a_{21}$ and (7.2) determines $a_{22}$ (given that it is real and non-negative). Note that the sum in (7.2) includes
all the terms (such as $a_{00}= ||A_0||$ and $a_{20} = -\alpha ||A_3||$, and so on). If $a_{11}=0$, then we set $a_{21}=0$
and again $a_{22}$ is determined.
\end{itemize}\end{proposition}

\begin{proof} The cases where $d=0$ and $d=1$ are immediate. When $d=2$, we have $A_0$ and $A_2$ are orthogonal and $A_2 \ne 0$.
We may therefore choose the first and third columns as claimed. By  (4) and (3) we also must have
$$ \langle A_0, A_1 \rangle + \langle A_1, A_2 \rangle = 0  \eqno (8.1) $$ 
$$ ||A_0||^2 + ||A_1||^2 + ||A_2||^2 = 1 \eqno (8.2) $$
Equation (8.1) becomes
$$ ||A_0|| {\overline {a_{10}}} + a_{12} ||A_3|| = 0 $$
which we solve in terms of $\alpha$. Then we satisfy (8.2) by choosing $|a_{11}|$ appropriately. Using a diagonal unitary map
with diagonal elements $1,e^{i\phi},1$ for appropriate $\phi$ we can make $a_{11}$ real and non-negative.

The case when $d=3$ is similar to the case when $d=2$ except that we must regard $a_{10}$ and $a_{23}$ as parameters.
We then use equations (4) to obtain (7.1) and (7.2). Using a diagonal unitary map we can make the parameter $a_{11}$ real and, after determining 
$|a_{22}|$, make $a_{22}$ real as well. The normal form follows. Note that there are six parameters, of which
$||A_0||, ||A_3||, a_{11}$ are real and $\alpha, a_{10}, a_{23}$ are complex. \end{proof}

Before proving Theorem 4.1, we discuss {\bf upper-trace identities}.

\begin{definition} Consider a matrix ${\mathcal V}$ whose column vectors are $V_0,...V_d$. 
We say that ${\mathcal V}$ satisfies {\it upper-trace identities} with parameter values $\lambda_j$ if

$$ \langle V_0, V_d \rangle = \lambda_d. $$

$$ \langle V_0, V_{d-1} \rangle + \langle V_1, V_d \rangle = \lambda_{d-1} $$

$$ \sum_j \langle V_j, V_{j+k} \rangle = \lambda_k $$

$$ \sum_j ||V_j||^2 = \lambda_0.  $$ \end{definition}

When $\lambda_0 =1$ and $\lambda_j=0$ otherwise, these identities are equations (3) and (4) from Proposition 4.1.
Their general version facilitates the proof of Theorem 4.1 as follows: 

\begin{proof} (of Theorem 4.1) The upper-trace identities are linear equations in the inner products, but quadratic in the entries
of the vectors after a basis has been chosen. By regarding some of the entries as parameters we obtain
linear equations in the remaining entries. Let $f$ be a polynomial of degree $d$ mapping
the unit circle to some unit sphere. Its coefficient vectors $A_j$ and an orthonormal basis determine
a matrix ${\mathcal A}$ with entries $a_{jk}$ for $0 \le j,k \le d$. Equation (3) is equivalent to
${\rm tr} \left ( {\mathcal A}^* {\mathcal A}\right)= 1$. When $d=2$ this equation is the same as (6) and when
$d=3$ it is the same as (7.2). We first find a unitary change of coordinates 
such that the first column of ${\mathcal A}$ is $||A_0|| e_0$ and the last column is $||A_d|| e_d$ for
orthonormal basis elements $e_j$. Once we have done so, we can only choose unitary maps
that preserve these vectors. 

Consider the identity $\langle A_0, A_{d-1}\rangle + \langle A_1, A_d \rangle = 0$.
There is thus a complex number $\alpha$ such that $a_{1d} = {\overline \alpha} ||A_0||$ and 
$a_{(d-1)0} = -\alpha ||A_d||$. The rest of the entries in the first and last row are complex parameters.
These parameters are of course subject to the inequalities forced upon us by equation (3) from Proposition 4.1.
The entries $a_{jk}$ for $1 \le j,k \le d-1$ define a submatrix of size $(d-1)$ by $(d-1)$, called the inner matrix. Let us call the column vectors
of this matrix $\beta_j$. For $k \ge 0$, the equations (4) now take the form
$$ \sum_k \langle \beta_k, \beta_{k+l} \rangle = \lambda_l, $$
where the complex numbers $\lambda_l$ are expressed in terms of the parameters 
used in the first and last column and the first and last
row of ${\mathcal A}$. Thus the inner $(d-1)$ by $(d-1)$ matrix satisfies upper-trace identities with known parameter values $\lambda_j$.

Given a matrix whose columns are the vectors $\beta_j$, we choose an orthonormal basis
putting it in upper triangular form, and furthermore making the diagonal entries real and non-negative. Doing so depends
only upon the inner products, and hence the upper-trace identities are preserved.
We have established an induction from degree $d-1$ with no constant term to degree $d$ possibly with a constant term.
Hence Theorem 4.1 follows from induction, using as basis steps the cases where the degrees are $1$ and $2$. \end{proof}

We next derive some corollaries, including information in the rational case.
Put 
$$G_{\alpha, r}=(\cos \alpha \frac{r-z}{1-rz} , \sin \alpha, 0,..., 0)$$ 
for $\alpha \in [0, \frac{\pi}{2})$ and $r \in [0,1)$.

\begin{proposition}\label{pdege}
For $N\ge 2$, let $F: \Delta \rightarrow \mathbb{B}_N$ be a holomorphic proper rational map of degree one. Then
\begin{enumerate}
\item $F$ is unitarily equivalent to some $G_{\alpha,r}$.

\item If $F$ is polynomial, then it is unitarily equivalent to some $G_{\alpha, 0}$.

\item For $\alpha, \beta \in [0, \frac{\pi}{2})$ and $0 \leq r_1, r_2 < 1$, the maps $G_{\alpha, r_1}$ and
$ G_{\beta, r_2}$ are unitarily equivalent if and only if $\alpha=\beta$ and $r_1=r_2$. 

\item A holomorphic rational proper map $F$ of degree one from $\Delta$ to $\mathbb{B}_N$ is spherically equivalent to $z \oplus 0$.
\end{enumerate}
\end{proposition}

\begin{proof}
Part (2) follows directly from Proposition 4.2. Part (3) is a simple computation.
We will therefore prove only parts (1) and (4). 

To prove part (1), start with $F(z) = {zA +B \over 1- {\overline c} z}$, where $A,B$ are vectors and $|c|<1$. 
After replacing $z$ by $e^{i\theta} z$, a unitary change in the source,  we may assume that $c=r$ is real.
We then rewrite the numerator for new vectors $A'$ and $B'$ as
$$ zA+B = (z-r)A' + (1-r z) \ B'. $$
Doing so is possible because the matrix $\begin{pmatrix} 1 & -r \cr -r & 1\end{pmatrix}$ is invertible. Then
$$ F(z) = {z-r \over 1- rz}A' + B'. $$
Note that $A'$ is multiplied by a sphere map. Since $F$ is a sphere map, $||F(z)||^2=1$ on the circle. 
It follows that
$||A'||^2 + ||B'||^2 = 1$ and $\langle A', B'\rangle = 0$.
Now we make a unitary change of coordinates in the target such that $A'=(\cos(\alpha),0, \dots, 0)$
and $B'= (0, \sin{\alpha}, 0, \dots, 0)$. Thus $F$ is unitarily equivalent to $G_{\alpha, r}$.

Finally we prove part (4). By composing $F$ with an automorphism of $\mathbb{B}_N$
if necessary, we may assume $F(0)=0$ and hence $F(z) = {Az \over (1-bz)}$ for some constant vector $A$.
By applying a unitary map, we may further assume
$$F(z)= {az \over 1-bz}\oplus 0. $$
Since $F$ is a proper map from $\Delta$ to $\mathbb{B}_N$, it follows that $\frac{az}{1-bz}$ 
is a proper map from $\Delta$ to $\Delta$. 
This fact forces $b=0$
and $a=e^{i\theta}$, and concludes the proof.
\end{proof}

We define a two-parameter family of maps by
$$ J_{\alpha,\beta} = C_2 z^2 + C_1 z + C_0$$
where
$$ C_2 = \left(\cos(\alpha) \cos(\beta), 0 \right) $$
$$ C_1 = \left( \sin(\alpha) \sin(\beta) ,-\sin(\alpha) \cos(\beta)  \right) $$
$$ C_0 = \left(0, \cos(\alpha) \sin(\beta)\right). $$
Here $\alpha, \beta \in [0, {\pi \over 2})$. Note that $\langle C_2, C_0 \rangle = 0$, as required by (4).

We define a three-parameter family of maps by

$$ J_{\alpha,\beta, \gamma} = (D_2 z^2 + D_1 z + D_0) \oplus 0$$
where
$$ D_2 = \left(\cos(\alpha) \cos(\beta), 0, 0 \right) $$
$$ D_1 = \left( \sin(\alpha) \sin(\beta) \sin(\gamma),\sin(\alpha) \cos(\gamma) , -\sin(\alpha) \cos(\beta) \sin(\gamma) \right) $$
$$ D_0 = \left(0,0, \cos(\alpha) \sin(\beta)\right). $$
Here $\alpha, \beta \in [0, {\pi \over 2})$ and $\gamma \in [0,{\pi \over 2}]$. 
Here $\langle D_2, D_0 \rangle = 0$, as required by (4).

The following result follows from our normal forms when the degree is $2$.

\begin{proposition}\label{degpr} Let $F: \Delta \rightarrow \mathbb{B}_N$ be a holomorphic polynomial proper map of degree two.
\begin{enumerate}

\item Assume $N=2$. Then $F$ is unitarily equivalent to some  $J_{\alpha, \beta}$.

\item For  $\alpha_1, \alpha_2, \beta_1, \beta_2$ in $[0, \frac{\pi}{2})$, 
the maps $J_{\alpha_1, \beta_1}$ and $J_{\alpha_2, \beta_2}$ are unitarily equivalent if and only if $\alpha_1 =\alpha_2$
and $\beta_1 = \beta_2$.

\item Assume $N\ge 3$. Then $F$ is unitarily equivalent to some $ J_{\alpha, \beta, \gamma}$.

\item For $\alpha_1, \alpha_2, \beta_1, \beta_2, \in (0, \frac{\pi}{2}), \gamma_1, \gamma_2 \in [0, \frac{\pi}{2})$, the maps $ J_{\alpha_1, \beta_1, \gamma_1}$ and $J_{\alpha_2, \beta_2, \gamma_2}$ are unitarily equivalent if and only if
$(\alpha_{1}, \beta_1, \gamma_1)=(\alpha_2, \beta_2, \gamma_2)$.
\end{enumerate}

\end{proposition}

\begin{remark}
The maps $J_{0, \beta, \gamma}, J_{\alpha, 0, \gamma}, J_{\alpha, \beta, \frac{\pi}{2}}$  in part (3) of Proposition \ref{degpr}
each have embedding dimension $2$. 

\end{remark}

\medskip We next illustrate the proof of Theorem 4.1 in the degree $4$ case. Write
$$ f(z) = A_0 + A_1 z + A_2 z^2 + A_3 z^3 + A_4 z^4. $$
The condition $\langle A_0, A_4 \rangle = 0$ enables us to choose coordinates such that
$$ {\mathcal A} = \begin{pmatrix} ||A_0|| & * & * & * & 0 \cr 0 & * & * & * & 0 \cr  0 & * & * & * & 0 \cr 0 & * & * & * & 0 \cr 0 & * & * & * & ||A_4|| \end{pmatrix} $$

Now use the condition $\langle A_0, A_3\rangle + \langle A_1, A_4 \rangle = 0 $ to get
$$ {\mathcal A} = \begin{pmatrix} ||A_0|| & ** & ** & -\alpha ||A_4|| & 0 \cr 0 & * & * & * & 0 \cr  0 & * & * & * & 0 \cr 0 & * & * & * & 0 \cr 0 & {\overline \alpha} ||A_0||  & ** &  ** & ||A_4|| \end{pmatrix} $$
Regard the ** as complex parameters. The rest of the identities become upper-trace identities on the inner three-by-three
matrix. Since everything depends only on the inner products, things are unitarily invariant.
We can then use a unitary to put the inner matrix in upper-triangular form:

$$ \begin{pmatrix} ||A_0|| & ** & ** & -\alpha ||A_4|| & 0 \cr 0 & * & * & * & 0 \cr  0 & 0 & * & * & 0 \cr 0 & 0 & 0 & * & 0 \cr 0 & {\overline \alpha} ||A_0||  & ** &  ** & ||A_4|| \end{pmatrix}. $$

The inner matrix may be regarded as $3$ column vectors $\beta_1, \beta_2, \beta_3$.
The upper-trace identities express
$$ ||\beta_1||^2 + ||\beta_2||^2 + ||\beta_3||^2   \eqno (9.1) $$
$$ \langle \beta_1, \beta_2 \rangle + \langle \beta_2, \beta_3 \rangle \eqno(9.2)$$
$$ \langle \beta_1, \beta_3 \rangle \eqno (9.3) $$
as known quantities in terms of the various parameters along the first and last row of ${\mathcal A}$.
There are $5$ complex and $2$ real parameters along the first and last row. There are $6$ unknown parameters
in the inner matrix, but these satisfy $3$ equations arising from expressing (9.1), (9.2), (9.3) in terms
of the known parameters. Hence there remain $3$ new parameters. By using diagonal unitaries we obtain in total
$10$ parameters of which $4$ are real.

\section{homotopy}

Let $\mathcal{H}(n,N)$ denote the  set of homotopy equivalence classes of rational proper
holomorphic maps from $\mathbb{B}_n$ to $\mathbb{B}_N$. We write $h(n,N)$ for the 
cardinality of $\mathcal{H}(n,N)$. We recall some basic facts
about these homotopy classes. Each class in $\mathcal{H}(1,1)$ is given by $z^k$ for a positive integer $k$. 
When $N\ge n\ge 2$, the cardinality  $h(n,N)$ is always finite. See [DL1].  Of course, when $N<n$, the set is empty.
Very few of the the $h(n,N)$ are known. We have $h(2,3)=4$.  When $n \geq 3,$ we have
$h(n, N)=1$ if $n \leq N \leq 2n-2$ and $h(n, 2n-1)=2$. If $n \geq 4$, then $ h(n, 2n)=1$. 
We next establish  the following new result along these lines.

\begin{theorem}\label{t1}
For $N\ge 2$, all holomorphic rational proper maps from $\Delta$ to $\mathbb{B}_N$  belong to one homotopy equivalence class. Thus $h(1,N)=1$ if $N \geq 2$.
\end{theorem}

\begin{lemma}\label{lemmah} Let $F$ and $G$ be homotopic proper maps from $\Delta$ to some ${\mathbb B}_N$.
Let $\zeta$ be a proper map from $\Delta$ to itself. Then $\zeta F$ and $\zeta G$ are
also homotopic proper maps from $\Delta$ to ${\mathbb B}_N$. 
Similarly, if $(f_1,...,f_N)$ and $(g_1,...,g_N)$ are homotopic via polynomial sphere maps
from $\Delta$ to $\mathbb{B}_N$, and $\zeta$ is a polynomial proper map, then $(f_1,...,f_{N-1}, \zeta 
f_N)$ and $(g_1,...,g_{N_1}, \zeta g_N)$
are homotopic via polynomial proper maps.
\end{lemma}

\begin{proof}
Suppose for $t \in [0,1]$ that $H_t$ is a family of proper maps, depending continuously on $t$,  with $H_0=F$ and $H_1=G$.
Then $\zeta H_t$ also depends continuously on $t$. Since $\zeta$ maps the circle to itself,
$||\zeta H_t||^2 = ||H_t||^2$ on the circle, and hence $\zeta H_t$ is also proper. The second statement follows from a similar argument, assuming $\zeta=z^m$. 
\end{proof}

\begin{lemma}\label{lemmadeg}
Let $F=\frac{p}{q}=\frac{(p_1,...,p_N)}{q}$ be a holomorphic rational proper map from $\Delta$ to $\mathbb{B}_N$ that is reduced to lowest terms with $F(0)=0$. Then $\mathrm{deg}(p) > \mathrm{deg}(q).$
\end{lemma}

\begin{proof} Assume $F$ has degree $d$. Put $p(z) = \sum_{j=0}^d P_j z^j$ and $q(z) = \sum_{j=0}^d q_j z^j$. 
The condition $||p(z)||^2 = |q(z)|^2$ on the sphere yields upper-trace identities
as in Proposition 4.1. Assuming that $q(0)=1$, one of these identities
gives $\langle P_d, P_0 \rangle = q_d$.  Since $P_0$ is assumed to be $0$, the conclusion follows.
\end{proof}

\begin{proof} (of Theorem 5.1) Let $F$ be any holomorphic rational proper map from $\Delta$ to $\mathbb{B}_N$. It 
suffices 
to show that $F$ is homotopic to $z \oplus 0$. Write $d=\mathrm{deg}(F)$. We will prove the result by induction
on the degree.  Part 4 of Proposition 4.3 shows that the statement holds for $d=1$. We assume it holds for $d \leq k$ 
for some $k \geq 1$. Now consider $F$ of degree $d=k+1$. Note 
that $F$ is spherically equivalent and thus homotopic to a 
rational proper map $F_1$ with $F_{1}(0)=0$ and $\mathrm{deg}(F_1)=d=k+1$. 
Put $F_1=zF_2$. It follows from Lemma \ref{lemmadeg} that $\mathrm{deg}(F_2) = k.$ By
 the induction hypothesis, $F_2$ is homotopic to $z \oplus 0$. 
We conclude by Lemma \ref{lemmah} that $F_1=zF_2$ is 
homotopic to $z^2 \oplus 0$ which is also homotopic to $z \oplus 0$. Theorem \ref{t1} follows.
\end{proof}

\bigskip

We write $\mathcal{S}(N,d)$ for the solution set of the system from Proposition 4.1:
$$\mathcal{S}(N, d)=\{(A_0,...,A_d):~\text{equations (3) and (4) are satisfied} \}.$$
Let $\mathcal{P}(N,d)$ be the collection of all holomorphic polynomial sphere maps  from $\Delta$ to $\mathbb{C}^N$
of degree at most $d$. There is a one-to-one correspondence between
$\mathcal{S}(N, d)$ and $\mathcal{P}(N, d)$.

The following result shows that $\mathcal{S}(N, d)$ is a connected subvariety of
$\mathbb{C}^{N(d+1)}$ when $N \geq  2$.

\begin{proposition}\label{prpl}
Assume $N \geq  2$. Any two polynomial sphere maps from $\Delta$ to $\mathbb{C}^N$
of degree $d$ are homotopic via polynomial sphere maps.
\end{proposition}

\begin{proof}
The conclusion follows if $\mathcal{P}(N, d)$ is path-connected. We prove the
path-connectedness in the next lemma.
\begin{lemma}\label{lemma24}
$\mathcal{P}(N,d)$ is path-connected for any $N \geq 2$ and $d \geq 0$.
\end{lemma}

{\it Proof}. We prove the lemma by induction on $d$. Proposition 4.2 shows that
the lemma holds for $d=0$ and $d=1$. We assume for some $k \geq 1$ that it is true for $d \leq k$ for
and we will prove it for $d=k+1$. Fix  $f \in \mathcal{P}(N, d)$ with $d=k+1$. 
Write $f=\sum_{j=0}^d  A_j z^j$ as usual; here we regard the $A_j$ as column vectors.
Note that $f$ has $N$ components. It suffices to show $f$ is homotopic to $z \oplus 0$ in $\mathcal{P}(N,d)$.  Write
$$ A=\left( A_0,\dots, A_d \right).$$

We may assume $f$ is of degree $d$ and hence $A_d \neq 0$.  Find a unitary matrix $U$ such that
$UA_d= (0,...,0, ||A_d||)$. Put $B=UA=(B_0,...,B_d)$ and set $g=Uf$.
Since the unitary group is path-connected, $g$ is homotopic to $f$. Moreover,
by the form of $B_d$, each $g_j$ is of degree at most $d-1=k$ except for $g_N$ which is of
degree $d$.
We still have $\langle B_0, B_d \rangle =0$.  Hence the last
component of $B_0$ is zero and thus $g_N$
has no constant term. Write $g_N=zu_N$. We set
$$ h=(g_1,...,g_{N-1}, u_N).$$
Note that  $h \in \mathcal{P}(N,d-1)$. It follows from the induction
hypothesis that $h$ is homotopic to $z \oplus 0$ in $\mathcal{P}(N,d-1)$. It follows
from Lemma \ref{lemmah} that
$g$ (and thus $f$) is homotopic to $z^2 \oplus 0$ and thus to $z \oplus 0$ in
$\mathcal{P}(N,d)$. Lemma \ref{lemma24} follows.
As noted above, the proposition follows from this lemma.
\end{proof}

We naturally ask the following question in the general setting.

\medskip

{\bf Question:} {\it Assume $f,g: \mathbb{B}_n \to \mathbb{B}_N$ are homotopic polynomial proper maps.
Can we always find a homotopy where each intermediate map is also a polynomial?}

\medskip 

The answer is affirmative in some special cases. If $3 \le n \le N \le 2n$, for example, then all the polynomial
maps are known and the result holds. When $n=2$ and $N=4$, there are many more polynomial maps
and the precise homotopy relations are not known.

\section{bibliography}

[BEH] M. S. Baouendi, P. F. Ebenfelt, X. Huang, Holomorphic mappings between hyperquadrics with small
signature difference, {\it  Amer. J. Math.} 133 (2011)， 1633-1661.

[CS] J. A. Cima and T. J. Suffridge, Boundary behavior of rational proper maps. {\it Duke Math. J.} 60 (1990), no. 1, 135-138.

\medskip

[D1] J. P. D'Angelo,  Several Complex Variables and the Geometry of Real Hypersurfaces,
CRC Press, Boca Raton, Fla., 1992.

\medskip
[D2] J. P.  D'Angelo, Proper holomorphic mappings,
positivity conditions, and isometric imbedding, {\it J. Korean Math Society}, May 2003, 1-30.

\medskip

[DL1] J. P. D'Angelo and J. Lebl, Homotopy equivalence for proper holomorphic mappings, {\it Adv. Math.} 286 (2016), 160-180.

\medskip

[DL2] J. P. D'Angelo and J. Lebl, On the complexity of proper mappings between balls, {\it Complex Variables and Elliptic Equations},
Volume 54, Issue 3, Holomorphic Mappings (2009), 187-204.

\medskip

[EHZ] P. F. Ebenfelt, X. Huang, D. Zaitsev, The equivalence problem and rigidity for hypersurfaces
embedded into hyperquadrics,
{\it Amer. J. Math.} 127 2005, 169-191.

\medskip

[F] F. Forstneric,
Extending proper holomorphic maps of positive codimension,
{\it Inventiones Math.}, 95(1989), 31-62.

\medskip
[H] X. Huang, On a linearity problem for proper maps between balls in complex spaces
of different dimensions, {\it J. Diff. Geometry} 51 (1999), no 1, 13-33.

\medskip
[HJ] X. Huang, X., and S. Ji, 
Mapping ${\bf B}_n$ into ${\bf B}_{2n-1}$, {\it Invent. Math.} 145 (2001), 219-250. 

\medskip
[JZ] S. Ji and Y. Zhang,
Classification of rational holomorphic maps from $B^2$ into $B^N$ with degree $2$.
{\it Sci. China Ser. A } 52 (2009), 2647-2667.

\medskip

[L] J. Lebl, Normal forms, Hermitian operators, and CR maps of spheres and hyperquadrics,
{\it  Michigan Math. J.} 60 (2011), no. 3, 603-628.

\end{document}